\documentclass{amsart}
\usepackage[utf8]{inputenc}
\usepackage[margin=1.5in]{geometry}

\usepackage{hyperref}
\hypersetup{
    colorlinks=true,
    citecolor=blue,
    linkcolor=blue,
    filecolor=magenta,      
    urlcolor=cyan,
    pdftitle={Homological Skein Modules},
    pdfpagemode=FullScreen,
    }

\usepackage{tikz-cd}

\usepackage{hyperref}
\hypersetup{
    colorlinks=true,
    linkcolor=purple,
    citecolor=purple,      
    urlcolor=purple,
}
\usepackage{amsmath,amssymb,amsthm}
\usepackage{tikz}
    \usetikzlibrary{cd} 
\usepackage{tikz-cd}
\allowdisplaybreaks
\usepackage{comment}
\usepackage{todonotes}
\usepackage{stmaryrd} % for BNbracket
\usepackage{MnSymbol} % for Obrack
\usepackage{textcomp} % alternative for above in textmode
\usepackage{enumitem}

\usepackage{bbm}

\theoremstyle{definition}
\newtheorem{theorem}{{Theorem}}[section]
\newtheorem{lemma}[theorem]{{Lemma}}
\newtheorem{proposition}[theorem]{{Proposition}}

\newtheorem{corollary}[theorem]{{Corollary}}
\newtheorem{example}[theorem]{{Example}}
\newtheorem{conjecture}[theorem]{{Conjecture}}

\theoremstyle{definition}
\newtheorem{remark}[theorem]{Remark}
\newtheorem{definition}[theorem]{{Definition}}

\newcommand{\KhR}{\mathrm{KhR}} % Kh-Roz
\newcommand{\skein}{\mathcal{S}}
\newcommand{\colim}{\mathrm{colim}\,}

\usepackage{todonotes}

\title[Bar-Natan lasagna modules and exotic surfaces]{Bar-Natan skein lasagna modules and exotic surfaces in 4-manifolds}
\author{Ian A. Sullivan}
\date{\today}

\begin{document}

\begin{abstract}
    We construct and study the skein lasagna module obtained by importing the Bar-Natan Khovanov homology package. For 4-manifolds satisfying a non-vanishing condition, we produce pairs of exotic surfaces (with boundary) by using the behavior of skein lasagna gluing maps associated to connect sums of 4-manifolds. We show that one internal stabilization is generally not enough for these exotic knotted surfaces, generalizing results of Hayden to 4-manifolds that contain homologically diverse surfaces admitting primitive fillings.
\end{abstract}

\maketitle

\tableofcontents

\section{Introduction}
The skein lasagna module construction of Morrison--Walker--Wedrich \cite{MWW-lasagna} has proven useful for the production of smooth 4-manifold invariants. The skein lasagna perspective provides a framework for extending link invariants that arise in the link homology setting to invariants of smooth 4-manifolds. For example, the lasagna $s$-invariant of Ren--Willis \cite{RenWillis}, as an extension of Rasmussen's $s$-invariant \cite{rasinv-genus}, is an invariant of smooth $4$-manifolds capable of detecting exotic phenomena.

In this note, we import the $\mathbb{F}_{2}[H]$ link homology theory of \cite{BN-tangles}, often referred to in the literature as \emph{Bar-Natan} homology, and extend invariant properties of this link homology theory to the smooth 4-manifold setting. We study the naturally arising $\mathbb{F}_{2}[H]$-module structure of this invariant, which we denote by $\skein_{0}^{BN}(X;L)$ for a 4-manifold and boundary link pair $(X,L)$. We interpret the effect of multiplication by $H$ as a 3-dimensional 1-handle surgery in $X$ corresponding to attaching a handle to the skein surfaces representing elements of $\skein_{0}^{BN}(X;L)$. Using this interpretation, we study the $H$-torsion elements in this Bar-Natan lasagna setting, and use a \emph{lasagna $H$-torsion order} to extend the internal stabilization results of Hayden \cite{hayden2023atomic}.

Using results about a more general form of this invariant, defined and explored in \cite{mwwsurfaces}, we are able to produce examples of exotically knotted pairs of surfaces in 4-manifolds other than $B^{4}$, that remain exotic after a single \emph{internal stabilization}.  

For a 4-manifold and boundary link pair $(X,L)$, a surface $S$ in $(X,L)$ is called \emph{homologically diverse} if no nontrivial union of the components of $S$ are null-homologous in $(X,L)$ (see Definition \ref{def:homdiv} - \cite{mwwsurfaces}). Given a Bar-Natan lasagna filling of $(X,L)$ represented by a homologically diverse surface, such a filling is called \emph{primitive} if the corresponding element $[S]\in\skein_{0}^{BN}(X;L)$ is not of the form $H\cdot{v}$ for some element $v\in{\skein_{0}^{BN}(X;L)}$ (see Definition \ref{def:primitive}).

\begin{proposition}(Proposition \ref{prop:main})
    Let $(F_{g},F^{\prime}_{g})$ be the pair of genus $g$ exotic surfaces constructed in \cite{hayden2023atomic} with boundary $K_{g}$ (see Figure \ref{fig:0}) and let $(X,L)$ be a 4-manifold and boundary link pair.  Suppose there exists a primitive filling $[S]\in \skein_{0}^{BN}(X;L)$, then there exists an exotically knotted pair $(G_{g},G^{\prime}_{g})$ of surfaces with local boundary $K_{g}$ in $X\setminus{B^{4}}$ that remain exotic after an internal stabilization. 
\end{proposition}

In Section \ref{sec:4}, we provide a method for determining if a Bar-Natan filling is primitive by using existing non-vanishing results for $\skein_{0}^{2}(X;L)$ over $\mathbb{F}_{2}$ coefficients. Furthermore, if the the original surface representing the skein surface of a primitive filling is known, then the exotically knotted pair $(G_{g},G_{g}^{\prime})$ may be constructed explicitly. In the end of Section \ref{sec:4}, we provide an example for the self intersection -1 sphere $\overline{\mathbb{C}P^{1}}$ in $\overline{\mathbb{C}P^{2}}$.

\begin{corollary}
    For the exotically knotted pair $(F_{g},F_{g}^{\prime})$ in Proposition \ref{prop:main}, the pair $(F_{g}\sqcup \overline{\mathbb{C}P^{1}}$, $F_{g}^{\prime}\sqcup \overline{\mathbb{C}P^{1}})$ is exotic in $(\overline{\mathbb{C}P^{2}}\setminus{B^{4}},K_{g})$, where $K_{g}$ is a boundary knot in the $S^{3}$ boundary of the removed 4-ball. Furthermore, one internal stabilization is not enough for this pair of surfaces.
\end{corollary}

Thus, the skein lasagna module construction detects exotically knotted surfaces. The invariant $\skein_{0}^{BN}(X;L)$ can also be used to show that exotically knotted pairs of local surfaces remain exotically knotted in 4-manifolds other than $B^{4}$. More generally, for exotically knotted pairs, the main proposition of this work implies the following.

\begin{corollary}
    Let $(\Sigma_{0},\Sigma_{1})$ be an exotically knotted pair of surfaces in $(B^{4},L)$ that induce distinct maps on Bar-Natan homology, and let $[S]$ be a primitive filling corresponding to a surface $S\subset{(X,L^{\prime})}$, then $(\Sigma_{0}\sqcup S)$ and $(\Sigma_{1}\sqcup S)$ are not smoothly isotopic in $(X\setminus{B^{4}};L\sqcup L^{\prime})$.
\end{corollary}

\textbf{Acknowledgments.}
The author would like to thank Melissa Zhang and Eugene Gorsky for their support and guidance. We also thank Qiuyu Ren, Mark Powell, and Trevor Oliveira-Smith for extremely helpful feedback, conversations, and comments on drafts of this work. 

\section{Preliminaries}
We begin by establishing our notion of exotic pairs of surfaces and internal stabilization, as well as the notation and conventions for the link homology tools that we use throughout this work.

\subsection{Topological preliminaries}

We are most interested in exotic phenomena in dimensions 2 and 4, specifically, phenomena involving exotically knotted pairs of surfaces in 4-manifolds. We recall that a pair of smooth 4-manifolds is called an \emph{exotic pair} if they are homeomorphic but not diffeomorphic. We recall also that there exists a natural number $k$ such that $X\#^{k}(S^{2}\times{S^{2}})$ and $Y\#^{k}(S^{2}\times{S^{2}})$ are diffeomorphic (see \cite{CTCWall}). We refer to the operation corresponding to taking a connect sum with a copy of $S^{2}\times{S^{2}}$ as an \emph{external stabilization} of the original 4-manifold. We turn now to surfaces in 4-manifolds. All manifolds (resp. submanifolds) are assumed to be smooth (resp. smoothly embedded) throughout.

\begin{definition}
    Let $\Sigma_{1}$ and $\Sigma_{2}$ be homologous surfaces with equal genus in some $4$-manifold $X$. Then $(\Sigma_{1},\Sigma_{2})$ form an \emph{exotic pair} of surfaces in $X$ if they are topologically isotopic (rel boundary if the surfaces have boundary), but not smoothly isotopic (rel boundary) in $X$. We call such surfaces \emph{exotically knotted} when they form an exotic pair.
\end{definition}

For any pair of exotically knotted surfaces, there exist topological operations that dissolve their exotic relationship.

\begin{definition}
    Let $\Sigma$ be a smoothly embedded surface in a 4-manifold $X$. Attaching a 1-handle to $\Sigma$ in $X$ while preserving orientability is called a \emph{weak internal stabilization} of $\Sigma$. If the surgery corresponds to taking a connect sum of pairs $(X,\Sigma)\#(S^{4},T^{2})$ where $T^{2}$ is an unknotted torus, then the operation is called a \emph{strong internal stabilization} of $\Sigma$. Throughout, we work entirely with strong internal stabilizations, simply referring to them as \emph{internal stabilizations}.
\end{definition}

For certain pairs of exotically knotted surfaces, the resulting surfaces after the application of sufficiently many internal stabilizations are smoothly isotopic rel boundary.

\begin{theorem}[\cite{Baykur2015}, Theorem 1]\label{thrm:baysun}
    Let $(\Sigma_{1},\Sigma_{2})$ be a pair of exotically knotted surfaces in $X$, then $\Sigma_{1}$ and $\Sigma_{2}$ become smoothly isotopic after some number $n\geq{0}$ of weak internal stabilizations. Furthermore, if the inclusion map $i:\partial(\nu \Sigma_{i})\hookrightarrow{X\setminus{\Sigma_{i}}}$ induces a surjection on $\pi_{1}$, then the result holds for strong internal stabilizations.
\end{theorem}

Note that the integer $n$ establishes a notion of \emph{stabilization distance} between exotically knotted pairs of surfaces. Theorem \ref{thrm:baysun} of Baykur and Sunukjian states that for exotically knotted pairs of surfaces $(\Sigma_{1},\Sigma_{2})$ satisfying the surjectivity on $\pi_{1}$ condition, there exists an integer $n>0$ such that the resulting surfaces $\Sigma_{1}\#^{n}T^{2}$ and $\Sigma_{2}\#^{n}T^{2}$ are smoothly isotopic. We remark also that exotically knotted surfaces also become smoothly isotopic after sufficiently many external stabilizations to their ambient 4-manifold. 

\begin{definition}
    Let $(\Sigma_{1},\Sigma_{2})$ be an exotically knotted pair of surfaces in $(B^{4},L)$ for some possibly empty link $L$. Define 
    \[d(\Sigma_{1},\Sigma_{2})=\min{\{k\;|\;\Sigma_{1}\#^{k}T^{2}\text{ and }\Sigma_{2}\#^{k}T^{2}\text{ are smoothly isotopic}\}}
    \]
    to be the \emph{internal stabilization distance} from $\Sigma_{1}$ to $\Sigma_{2}$.
\end{definition}

Baykur and Sunukjian show, for many families and constructions of exotically knotted surfaces, that $d(\Sigma_{1},\Sigma_{2})=1$. However, the \emph{internal stabilization conjecture} remains open and asks the following.

\begin{conjecture}\label{conj:intstab}
    If ($\Sigma_{1}$,$\Sigma_{2}$) is an exotically knotted pair of closed surfaces in $B^{4}$, then $d(\Sigma_{1},\Sigma_{2})=1$. 
\end{conjecture}

Although this question is very much open for closed surfaces, it is generally not true for surfaces with boundary in the 4-ball. For example, in \cite{hayden2023atomic}, Hayden constructs infinitely many exotically knotted pairs of surfaces with boundary, with arbitrarily high genus, such that one internal stabilization is not enough. Also, Guth in \cite{guth2022exotic} produces yet another infinite family of exotically knotted pairs of relative surfaces in the 4-ball with arbitrarily large internal stabilization distances. We remark also on the existence of closed exotically knotted pairs in 4-manifolds with boundary other than the 4-ball produced in \cite{hayden2023stabilizationclosedknottedsurfaces}. The result of Hayden is the most relevant to this work, so we record it now.

\begin{theorem}[\cite{hayden2023atomic}, Theorem A]\label{thrm:Hayden-A}
    For any $g\geq{1}$, there exist exotically knotted pairs of surfaces with boundary of genus $g$ in $B^{4}$ that remain exotic after one internal stabilization.
\end{theorem}

This theorem is proven by constructing pairs of surfaces that induce different maps on Bar-Natan homology. In particular, Hayden produces a knot $K_{H}$ (see Figure \ref{fig:0}), that bounds a pair of disks that are neither topologically nor smoothly isotopic, these disks yield a pair of genus-1 exotically knotted surfaces after attaching bands. Let these genus 1 surfaces be denoted $F_{1}$ and $F_{1}^{\prime}$ respectively, and let $K_{1}=\partial{F_{1}}=\partial{F^{\prime}_{1}}$ denote their boundary. We may then connect sum $F_{1}$ and $F_{1}^{\prime}$ with the fiber surface of $T_{2,2g-1}$ and we let the resulting exotically knotted pair be denoted $(F_{g},F^{\prime}_{g})$ with boundary denoted $K_{g}$ (see Figure \ref{fig:0}). We recall briefly the method used to prove Theorem \ref{thrm:Hayden-A} using Bar-Natan homology.

\begin{figure}
    \includegraphics[scale=0.75]{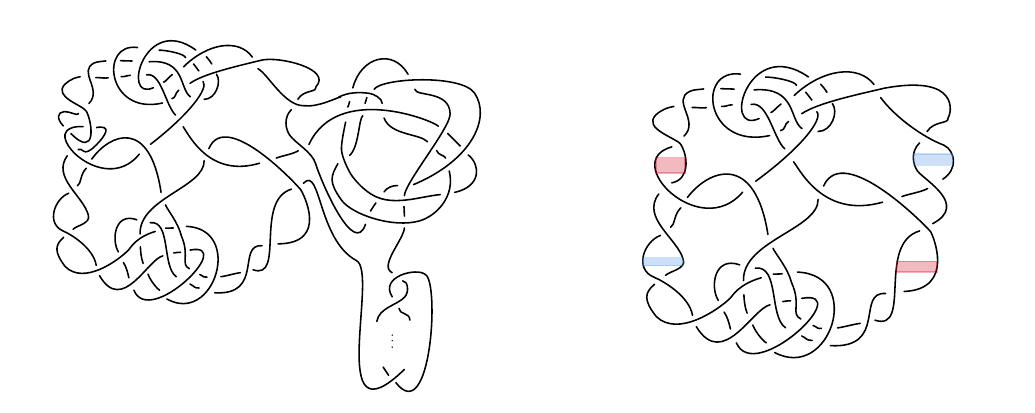}
    \caption{\textbf{Left:} A diagram of the knot $K_{g}=\partial{F_{g}}=\partial{F_{g}^{\prime}}$. \textbf{Right:} The knot $K_{H}$ with disks $D$ and $D^{\prime}$. Both diagrams are found in \cite{hayden2023atomic}.}\label{fig:0}
\end{figure}

\begin{remark}[\cite{hayden2023atomic}, Propositions 4.1 and 4.2]\label{prop:hay4.1,4.2}
    Let $-D$ and $-D^{\prime}$ denote the time-reversed (punctured) disks bounded by the knot $K_{H}$. These disks, taken as cobordisms $-D,-D^{\prime}:U\rightarrow{K_{H}}$, induce distinct maps on reduced Bar-Natan homology with $\mathbb{F}_{2}[H]$ coefficients. Furthermore, these maps remain distinct after multiplication by $H$, implying that the disks $D$ and $D^{\prime}$ remain exotic after a single internal stabilization.
\end{remark}

\subsection{Bar-Natan/equivariant link homology}

To establish preliminaries and conventions, as well as state Theorem \ref{thrm:Hayden-A} and Remark \ref{prop:hay4.1,4.2} in a more explicit way, we briefly review Bar-Natan's link homology theory, defined originally in \cite{BN-tangles}.

\begin{definition}\label{def:frobpair}
    The \emph{$U(2)$-equivariant Frobenius pair} is the pair $(R,A)$ where $R$ is the ground ring $\mathbb{F}_{2}[E_{1},E_{2}]$, and $A$ is the Frobenius algebra
    \[
        A:=\KhR_{U(2)}(unknot)=\frac{R[X]}{(X^{2}-E_{1}X+E_{2})}
    \]
    where $E_{1}$ and $E_{2}$ are classes of degree $2$ and $4$ respectively. We remark that the classes $E_{1}$ and $E_{2}$ can be identified with the first and second elementary symmetric polynomials in two variables $\{r_{1},r_{2}\}$, each of degree 2.
\end{definition}

By letting $H$ denote $E_{1}(r_{1},r_{2})$ with $r_{2}=0$, which implies that $E_{2}=0$, we obtain the Frobenius pair for Bar-Natan homology

\[
    R^{BN}=\mathbb{F}_{2}[H],\;\;A^{BN}=\frac{R^{BN}[X]}{(X^{2}-HX)}.
\]

Note that in these conventions, multiplication by $H$ is a bi-degree $(0,2)$ map. The link homology theory associated to the TQFT corresponding to the pair $(R^{BN},A^{BN})$ will be denoted $BN(\underline{\;\;})$ throughout; this is the theory typically referred to as \emph{Bar-Natan homology}. The local relations with dots and some consequences for Bar-Natan homology are given in Figure \ref{fig:localrels}.

\vspace{-5mm}

\begin{center}
    \begin{figure}[h]
        \includegraphics[scale=0.65]{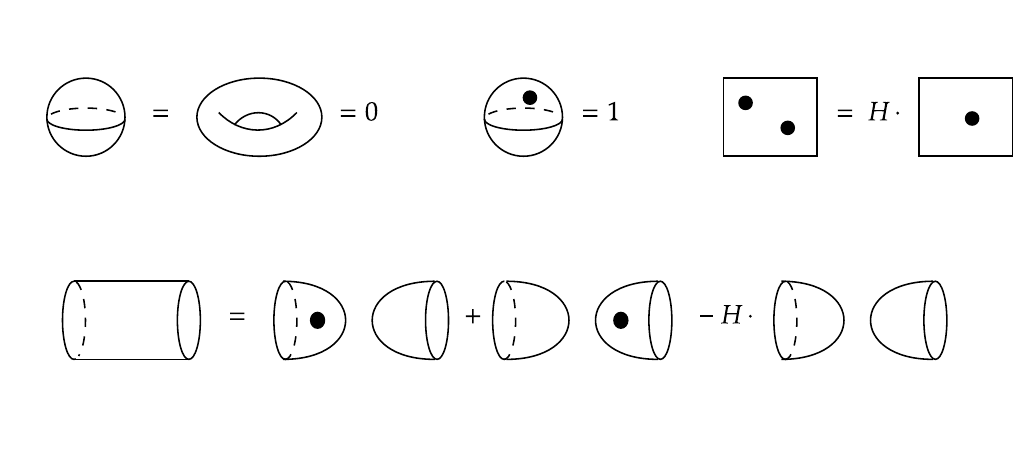}
        \vspace{-10mm}
        \caption{\textbf{Top:} The sphere, torus, dotted sphere, and $H$-trading relation over $\mathbb{F}_{2}$ coefficients. \textbf{Bottom:} the neck-cutting relation.}\label{fig:localrels}
    \end{figure}
\end{center}

Using the neck-cutting relation in Figure \ref{fig:localrels}, we can re-interpret the internal stabilization operation as in Figure \ref{fig:2}. We note that connect summing with $T^{2}$ is equivalent to the creation of a dotted torus, disjoint from the original surface.

    \begin{center}
        \begin{figure}
            \includegraphics[scale=0.7]{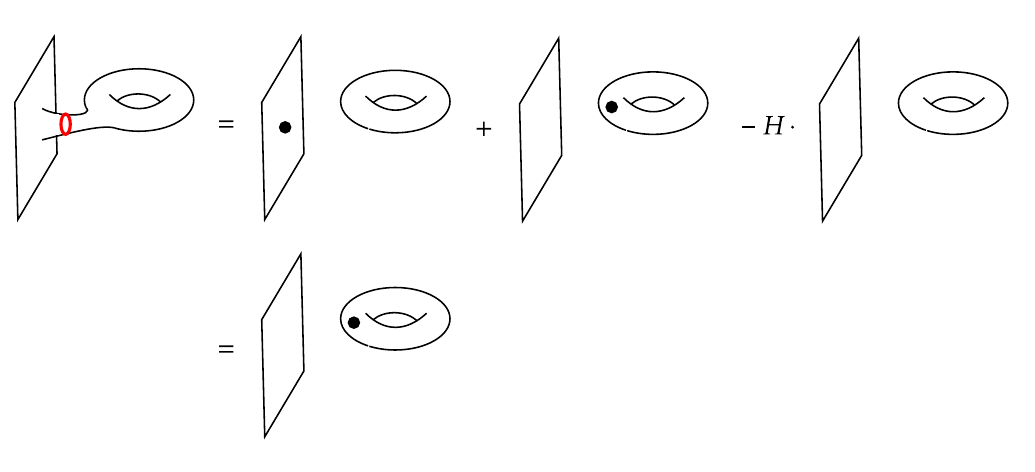}

        \vspace{-7mm}
        \caption{The effect of neck-cutting along the small red circle.}\label{fig:2}
        \end{figure}
    \end{center}

It will be necessary to discuss the reduced form of this theory, which we denote by $\widetilde{BN}(\underline{\;\;})$ in general, and $\widetilde{BN}_{y}(\underline{\;\;})$, $y\in{\{1,x\}}$ for a specific version of reduced Bar-Natan homology. For references, see \cite{immersedcurves,wigdersonBNsplits,AZ2022}.

\begin{definition}\label{def:relBN}
    Let $(L,p)$ be a link with a basepoint $p$. Let $f_{x}$ denote the chain map induced by the cobordism corresponding to the creation and merging of an $x$-labeled circle onto the component of $L$ containing $p$. We define the \emph{reduced Bar-Natan complex} $\widetilde{CBN_{x}}(L,p)$ as the image $\text{im}(f_{x})\subset{CBN(L)}$. We define $\widetilde{CBN_{1}}(L,p)$ to be the quotient complex $CBN(L)/\widetilde{CBN}_{x}(L)$. The homology of these complexes are isomorphic up to a shift in $q$-grading, and we denote the isomorphism type by $\widetilde{BN}(L,p)$.
\end{definition}

\begin{remark}
    Note that, due to the basepoint requirement, the reduced Bar-Natan homology theory is functorial for non-empty links and requires the introduction of punctures for cobordisms of the form $L\rightarrow{\emptyset}$ or $\emptyset\rightarrow{L}$. 
\end{remark}

We remark that a 1-handle surgery corresponding to an internal stabilization on a link cobordism $\Sigma:L_{0}\rightarrow{L_{1}}$ has the effect of multiplication by $2X-H$, or simply $H$ when working with $\mathbb{F}_{2}$ coefficients. In the following remark we record the more precise statement of Hayden. As per Remark \ref{prop:hay4.1,4.2}, we have that $\widetilde{BN}(-D)\neq{\widetilde{BN}(-D^{\prime})}$. Furthermore, we have that $H\cdot \widetilde{BN}(-D)\neq H\cdot \widetilde{BN}(-D^{\prime})$, implying that $-D$ and $-D^{\prime}$ do not become smoothly isotopic after a single internal stabilization.

\begin{remark}\label{rmk:deltaelement}
    Hayden proves this result by studying the element $\widetilde{\delta}_{0}:=\widetilde{BN}(-D)(1)-\widetilde{BN}(-D^{\prime})(1)\in{\widetilde{BN}(K_{H})}$. They show that $\widetilde{\delta_{0}}$ is non-zero and has $H$-torsion order greater than $1$. We now remark that a corresponding element in the unreduced theory also satisfies this property. This is required as we do not work in the reduced setting later on. To see this, we use a theorem of Wigderson \cite{wigdersonBNsplits}.
\end{remark}

\begin{theorem}[\cite{wigdersonBNsplits}, Theorem 3] \label{thrm:wigsplit}
    There is a natural $\mathbb{F}_{2}[H]$-module isomorphism
    \[
        \Phi_{K}:H_{*}(CBN(K))\xrightarrow{\cong}{H_{*}(\widetilde{CBN}_{x}(K))\oplus{H_{*}(\widetilde{CBN}_{1}(K))}}.
    \]
    Furthermore, there is an isomorphism $f_{K}:H_{*}(\widetilde{CBN}_{x}(K))\xrightarrow{\cong}{H_{*}(q^{-2}h^{0}\widetilde{CBN}_{1}(K))}$.
\end{theorem}

Equipped with this splitting isomorphism, we now have the following lemma.

\begin{lemma}\label{lem:unreducedworks}
    Let $\widetilde{\delta}_{0}\in{\widetilde{BN}_{x}}(K_{H})$ be the $H$-torsion element defined in Remark \ref{rmk:deltaelement} and let $\delta_{0}$ denote the element $\Phi_{K_{H}}^{-1}((\widetilde{\delta}_{0},0))\in BN(K_{H})$. The element $\delta_{0}$ is non-trivial with the same $H$-torsion order as $\widetilde{\delta}_{0}$.
\end{lemma}

\begin{proof}
    The result follows immediately from the definitions.
\end{proof}

\begin{definition}\label{def:localizedBN} (Inverting the discriminant) We now recall the definition of the \emph{localized} version of Bar-Natan homology. Let $D$ denote the discriminant of the polynomial $X^{2}-E_{1}X+E_{2}$, then define a Frobenius pair by $R_{D}:=R[D^{-1}]$ and $A_{D}=A\otimes_{R}{R_{D}}$. The TQFT associated to this Frobenius pair defines the localized $U(2)$-equivariant homology. In our case, after setting $E_{1}=H$ and $E_{2}=0$, we have that the discriminant is $H^{2}$ and we obtain an equivalent theory if we let $D=H$. Finally, after inverting the discriminant as above, we let $BN_{H=1}$ denote the link homology theory obtained from setting $H=1$ in the localized Bar-Natan theory. We remark that the link homology theory associated to this pair $(R^{BN}_{H^{-1}},A^{BN}_{H^{-1}})$ with $H=1$ is equivalent to the Lee homology theory \cite{lee-endo,khovanov2020linkhomologyfrobeniusextensions}. 
\end{definition}

\begin{remark}
We now claim that the Proposition of Hayden in Remark \ref{prop:hay4.1,4.2} can be interpreted and subsequently extended by passing to the \emph{homological skein module} (\emph{lasagna module}) setting. Before we perform this extension, we provide a brief overview of the invariants at play and their relevant features.
\end{remark}

\section{A skein lasagna module for mod 2 Bar-Natan homology}

The skein lasagna module recipe given by \cite{MWW-lasagna} allows the user to import any sufficiently well behaved functorial link homology theory to produce an invariant of smooth 4-manifolds. 

\begin{definition}
    Let $Z$ denote such a well-behaved TQFT for links in $\mathbb{S}^{3}$, and let $(X,L)$ denote a 4-manifold and boundary link pair. A \emph{skein surface} of $(X,L)$ consists of the topological data $(S,\{B_{i}\}_{i\in{I}},\{L_{i}\}_{i\in{I}})$, where $S$ is a properly embedded, oriented, framed surface such that $\partial{S}\cap{\partial{X}}=L$, where $\{B_{i}\}_{i\in{I}}$ is a finite collection of disjoint embedded $4$-balls with their interiors deleted from $\text{int}(X)$, such that $\partial{S}$ intersects the boundary of each $B_{i}$ in the link $L_{i}$. Skein surfaces are upgraded to \emph{$Z$-lasagna fillings} when the input links $\{L_{i}\}_{i\in{I}}$ are labeled by elements of $Z(L_{i})$ for each $L_{i}$. For a skein surface $S$ with label $v$ we use the notation $(S,v)$ to denote a corresponding $Z$-lasagna filling. The \emph{skein lasagna module} is
    \[
        \skein_{0}^{Z}(X;L):=\mathbb{F}\langle{\text{$Z$-lasagna fillings of $(X,L)$}}\rangle\big/\sim
    \]
where $\sim$ is the \emph{skein equivalence relation} defined in \cite{MWW-lasagna}. When say two lasagna fillings $(S,v)$ and $(S^{\prime},v^{\prime})$ are \emph{skein equivalent} if $(S,v)\sim{(S^{\prime},v^{\prime}})$.
\end{definition}

For a review of skein lasagna modules, we refer the reader to some of the many resources that exist in the literature: \cite{MWW-lasagna,MN22,RenWillis,SZ2024}. Throughout this note it will be necessary to distinguish between skein surfaces, $Z$-lasagna fillings, and skein equivalence classes of $Z$-lasagna fillings. For a lasagna filling $(S,v)$, we use the notation $[(S,v)]$ to denote the corresponding element in $\skein_{0}^{Z}(X;L)$; if the label is not relevant, we will drop it from the notation and write $[S]$ instead. When context is clear, we refer to lasagna elements $[S]$ simply as fillings.

Armed with this recipe, one may define a smooth 4-manifold invariant $\skein_{0}^{BN}$ by importing the Bar-Natan TQFT given by $(R^{BN},A^{BN})$. This invariant is not brand new, as the authors in \cite{mwwsurfaces} construct such an invariant for the more general $GL(2)$-equivariant version of $\mathfrak{gl}_{2}$ link homology. The invariant we are interested in is obtained from the construction in \cite{mwwsurfaces} by setting $N=2$, the ground ring $R=\mathbb{F}_{2}$, and the parameters $E_{1}$ and $E_{2}$ to $H$ and 0 respectively.

In \cite{mwwsurfaces}, the authors verify that the $GL(2)$-equivariant version satisfies the requirements for the construction of a skein lasagna module delineated in \cite{MWW-lasagna}, so the invariant $\skein_{0}^{BN}$ exists and is well-defined for Bar-Natan homology. Furthermore, since $\skein_{0}^{BN}$ is a specialization of the $GL(2)$-equivariant version, it enjoys some analogous features.

\begin{lemma}\label{lem:4balllasagnaisZ}
    For a link homology theory $Z$, the skein lasagna module construction for a pair $(B^{4},L)$ recovers $Z$. In particular,
    \[
    \skein_{0}^{Z}(B^{4};L)=Z(L).
    \]
\end{lemma}

We will make use of this property in Section \ref{sec:4} for $Z=BN(\underline{\;\;})$.

\begin{definition}\label{def:homdiv}
    A surface $S\subset{(X,L)}$ is called \emph{homologically diverse} if no non-trivial union of its closed components is null-homologous in $X$.
\end{definition}

\begin{remark} Simple examples of such surfaces include surfaces with boundary with no closed components, or closed homologically essential surfaces with a single closed component. The authors of \cite{mwwsurfaces} prove that homologically diverse surfaces correspond to non-trivial free generators of the $GL(2)$-equivariant skein lasagna module. In particular, the authors show that a homologically diverse surface $S\subset{(X,L)}$ defines a class in tri-degree $([S],\deg_{t}(S),\deg_{q}(S))$, where $[S]\in{H_{2}^{L}(X)}$, that is non-trivial modulo torsion in the $GL(2)$-equivariant skein lasagna module.
\end{remark}

In Corollary \ref{cor:divsurfacesforF2}, we verify that this result holds for our invariant over $\mathbb{F}_{2}$ coefficients. We will return to these surfaces and their significance after establishing some properties of the invariant $\skein_{0}^{BN}$.

We now define the $H$-action on Bar-Natan lasagna modules. The invariant $\skein_{0}^{BN}$ inherits an $\mathbb{F}_{2}[H]$-module structure from Bar-Natan homology.
We begin studying the this $\mathbb{F}_{2}[H]$-module structure by noting that the cabling directed system perspective of Manolescu-Neithalath \cite{MN22} applies to $\skein_{0}^{BN}$.

\begin{lemma}
    Let $(X,L)$ be a 4-manifold and boundary link pair, where $X$ is a 2-handlebody with Kirby diagram given by the framed oriented link $K$. Then there is an isomorphism
    \[
        \colim_{r\in{\mathbb{Z}_{\geq{0}}}}BN(K(r+\alpha^{+},r-\alpha^{-})\cup{L})^{B_{r+\alpha^{+},r-\alpha^{-}}}\{2r-|\alpha|\}\xrightarrow{\cong}\skein_{0}^{BN}(X;L,\alpha)
    \] 
    where $BN(K(r+\alpha^{+},r-\alpha^{-})\cup{L})^{B_{r+\alpha^{+},r-\alpha^{-}}}\{2r-|\alpha|\}$ denotes the Bar-Natan homology group of the cable $K(r+\alpha^{+},r-\alpha^{-})$, symmetrized with respect to the braid group action permuting the components of $K(r+\alpha^{+},r-\alpha^{-})$. The colimit is taken along annuli with (possibly zero) dots (or, equivalently, $H^{k}$ times a dotted annulus for $k\in{\mathbb{Z}_{\geq{0}}}$ as per the $H$-trading relation (Figure \ref{fig:localrels})).
\end{lemma}

\begin{proof}
    The proof of this lemma follows identically from the arguments presented in \cite{MN22} and \cite{MWWhandles}, with Bar-Natan homology taking the place of $\KhR_{2}$.
\end{proof}

\begin{remark}
    Over $\mathbb{F}_{2}$, this colimit perspective may 
    prove difficult to use for any direct computations, with difficulty largely attributed to the appearance of braid group representations of modules with $\mathbb{F}_{2}$ coefficients. There is no issue with choosing coefficients in a field more amicable to symmetrization, such as $\mathbb{Q}$, but this would come at the cost of the topological interpretation of the $H$-action.
\end{remark}

\subsection{Lasagna H-action for $(X,L)$} For reasons mentioned previously, we work over $\mathbb{F}_{2}$ for the remainder of this note. For a 4-manifold and boundary link pair $(X,L)$ and for a surface $S\subset{(X,L)}$, we adopt the notation $\skein_{0}^{Z}(\partial{X}\times{I};S):\skein_{0}^{Z}(X;L)\rightarrow{\skein_{0}^{Z}(X;L)}$ to denote the map defined by gluing the surface $S$ to fillings of $(X,L)$ along $L$.

\begin{figure}
    \begin{center}
        \includegraphics[scale=0.75]{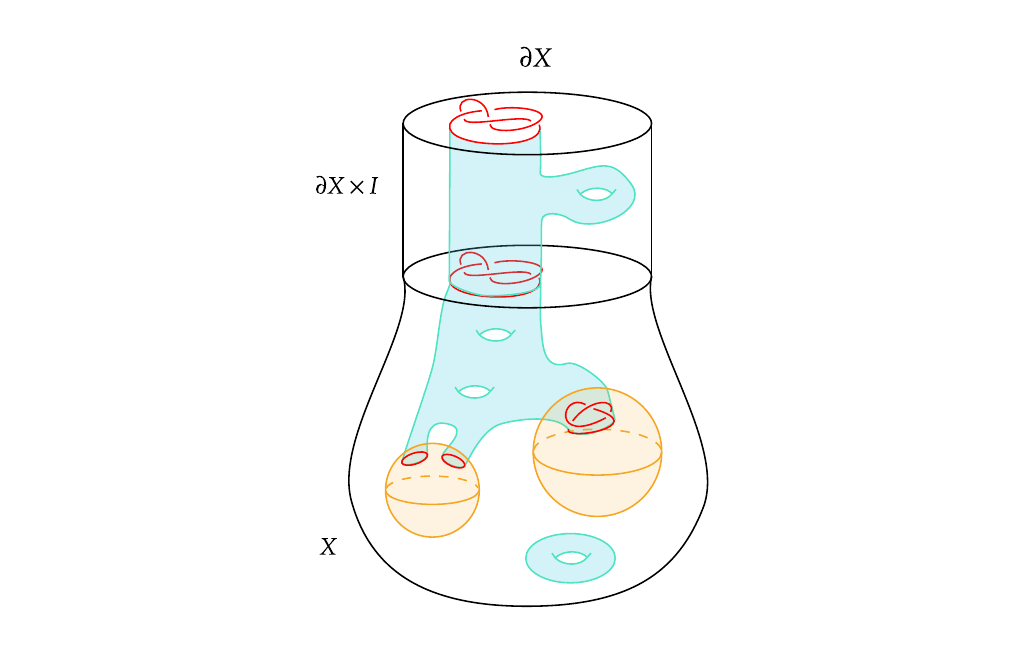}
    \end{center}
    \caption{The $H$-action map for $L\neq{\emptyset}$ as in Definition \ref{def:H-action}.}\label{fig2:H-action}
\end{figure}

\begin{definition}\label{def:H-action}
    Given a 4-manifold and boundary link pair $(X,L)$ where $L$ is not the empty link, we can interpret multiplication by $H$ as the following endomorphism on $\skein_{0}^{BN}(X;L)$
    \[
        H\cdot\underline{\;\;}\;:=\skein_{0}^{BN}(\partial{X}\times{I};T^{2}\#(L\times{I})):\skein_{0}^{BN}(X;L)\rightarrow{\skein_{0}^{BN}(X;L).}
    \]

Depicted in Figure \ref{fig2:H-action}, this $H$-action on $\skein_{0}^{BN}(X;L)$ has the topological effect of attaching an unknotted handle to a connected component of the skein surface of a Bar-Natan lasagna filling. Alternatively, if $L=\emptyset$, we realize the $H$-action as the creation of a local unknotted dotted torus in $\text{int}(X)$ by a neck-cutting on the skein surface (see Figure \ref{fig:2}), followed by an isotopy moving the dotted torus to the boundary of $X$. If $X$ is closed, we can remove an embedded $B^{4}$ to create such a boundary with no effect on $\skein_{0}^{BN}$. Letting $T^{2}_{\bullet}$ denote the dotted torus, we may define
\[
 H\cdot{\underline{\;\;}}:=\skein_{0}^{BN}(\partial X\times{I};T^{2}_{\bullet}):\skein_{0}^{BN}(X;\emptyset)\rightarrow{\skein_{0}^{BN}(X;\emptyset)}
\]

with $X$ replaced by $X\setminus{B^{4}}$ above in the case where $\partial X=\emptyset$. 

\end{definition}

With the $H$-action established, we define an analogue of $H$-torsion order.

\begin{definition}
    The \emph{lasagna H-torsion order} of a Bar-Natan lasagna element $\mathcal{F}\in{\skein_{0}^{BN}(X;L)}$ is
    \[
        \text{ord}_{H}^{(X,L)}(\mathcal{F})=\min{\{k\in{\mathbb{Z}_{\geq{0}}\;|\;H^{k}\cdot \mathcal{F}=0\}}}.
    \]
\end{definition}

When $X=B^{4}$, the lasagna $H$-torsion order and the usual $H$-torsion order of an element in Bar-Natan homology coincide. Hence, this definition is an extension of the usual notion of Bar-Natan homology $H$-torsion order (see \cite{Alishahi_2019Mathsci, gujral2020ribbondistanceboundsbarnatan} for more details and applications of this invariant for surfaces in $B^{4}$) for links in $S^{3}$ to surfaces in pairs $(X,L)$.

We complete this subsection by relating the invariant $\skein_{0}^{BN}$ to the $\KhR_{2}$ version. Recall that we recover the original Khovanov homology construction by setting $H=0$ in the Bar-Natan theory. 

\begin{lemma}\label{lem:linmapBNtoKhR} 
    There exists a linear map 
    \[
        F:\frac{\mathbb{F}_{2}[H]}{H}\otimes\skein_{0}^{BN}(X;L)\rightarrow{\skein_{0}^{2}(X;L;\mathbb{F}_{2})}
    \]
    obtained by applying a natural transformation $\frac{\mathbb{F}_{2}[H]}{H}\otimes BN(\underline{\;\;})\Rightarrow{\KhR_{2}(\underline{\;\;})}$
    to the labels on input links. The map $F$ is $H_{2}^{L}(X)\times{\mathbb{Z}_{t}}$-bigrading preserving. 
\end{lemma}

\begin{proof}
    Bar-Natan homology and $\KhR_{2}$ are related by tensoring with the module $\frac{\mathbb{F}_{2}[H]}{H}$ on the chain complex level. The universal coefficient theorem yields the corresponding natural transformation $\frac{\mathbb{F}_{2}[H]}{H}\otimes BN(\underline{\;\;})\Rightarrow{\KhR_{2}(\underline{\;\;})}$. This natural transformation respects the lasagna skein relation and therefore induces a morphism of skein modules.
\end{proof}

This lemma is an instance of Lemma 3.17 and Proposition 3.18 in \cite{mwwsurfaces}. It is also important to note that the natural transformation from $\frac{\mathbb{F}_{2}[H]}{H}\otimes{BN(\underline{\;\;})}$ to $\KhR_{2}(\underline{\;\;})$ is given by the universal coefficient theorem after tensoring with $\frac{\mathbb{F}_{2}[H]}{H}$ on the level of complexes.

\subsection{A connect sum gluing map for Bar-Natan skein lasagna} In order to extend exotically knotted pairs of surfaces in the 4-ball to other smooth 4-manifolds, it is necessary to understand the behavior of $\skein_{0}^{BN}$ on connect sums of 4-manifolds. Since we do not work over a field, this invariant does not admit a connect sum formula in the same manner as $\skein_{0}^{2}$ over field coefficients. Instead, we opt to study properties of the gluing map corresponding to the connect sum operation. Letting $X=X_{0}\cup_{Y}X_{1}$ where $Y$ is a properly embedded separating 3-manifold, there is a map $\phi:\skein_{0}^{BN}(X_{0}\sqcup{X_{1}};L_{0}\sqcup{L_{1})}\rightarrow{\skein_{0}^{BN}(X;L_{0}\sqcup{L_{1}})}$ corresponding to gluing $X_{0}$ to $X_{1}$ along $Y$ (for simplicity we assumed boundary links $L_{i}\subset{\partial{X_{i}}}$ do not intersect $Y$, although this is allowed in general, see \cite{MWWhandles}). Let
\[\phi_{\#}:\bigoplus_{L\subset{S^{3}}}\skein_{0}^{BN}(X_{0}\sqcup X_{1};L_{0}\sqcup L_{1}\sqcup L)\rightarrow \skein_{0}^{BN}(X_{0}\# X_{1};L_{0}\sqcup L_{1})
\]

denote the surjective homomorphism given by taking the connect sum of $X_{0}$ and $X_{1}$ with direct sum taken over links in the boundary of the connect sum $B^{4}$. Note also that there exists an injective homomorphism

\[
    \iota:\skein_{0}^{BN}(X_{0};L_{0})\otimes \skein_{0}^{BN}(X_{1};L_{1})\hookrightarrow \skein_{0}^{BN}(X_{0}\sqcup X_{1};L_{0}\sqcup L_{1})
\] given by applying the natural injective map $\mu:BN(K)\otimes{BN(K^{\prime})}\rightarrow{BN(K\sqcup K^{\prime})}$ to the input link labels. 

\begin{definition}
    Given 4-manifold and boundary link pairs $(X_{0},L_{0})$ and $(X_{1},L_{1})$ with skein surfaces $S_{0}$ and $S_{1}$ respectively. For each $S_{i}$, we let $K_{i}$ denote the input link after merging the input balls into a single $B^{4}$. We let $[(S_{0}\sqcup S_{1},v)]$ denote the skein equivalence class of the filling given by $(S_{0}\sqcup S_{1},v)$. We define the \emph{half torsion-free submodule} $\mathcal{R}(X_{0}\#X_{1};L_{0}\sqcup L_{1})$ of $\skein_{0}^{BN}(X_{0}\# X_{1};L_{0}\sqcup L_{1})$ as the submodule

    \[
        \mathcal{R}(X_{0}\# X_{1};L_{0}\sqcup L_{1}):=\mathbb{F}_{2}[H]\Big\langle [(S_{0}\sqcup S_{1},v)]\;|\;\text{Tor}_{1}^{\mathbb{F}_{2}[H]}(BN(K_{0}),BN(K_{1}))=0 \Big\rangle.
    \]
\end{definition}

In other words, $\mathcal{R}(X_{0}\# X_{1};L_{0}\sqcup L_{1})$ is the submodule generated by elements in $\skein_{0}^{BN}(X_{0}\# X_{1};L_{0}\sqcup L_{1})$ represented by a disjoint union of fillings whose Bar-Natan homology groups of their respective input links have pairwise vanishing torsion. We work with $\mathcal{R}(X_{0}\#X_{1};L_{0}\sqcup L_{1})$ to avoid the issues with torsion later on.

\begin{proposition}\label{prop:BNneckcut}
    The composition of maps $(\phi_{\#}\circ \iota)$ is invertible on $\mathcal{R}(X_{0}\# X_{1};L_{0}\sqcup L_{1})$. In particular, suppose that $[S_{0}\sqcup S_{1}]\in{\mathcal{R}}(X_{0}\# X_{1};L_{0}\sqcup L_{1})$, then
    \[
        \text{ord}_{H}^{(X_{0}\#X_{1},L_{0}\sqcup{L_{1}})}([S_{0}\sqcup{S_{1}}])=\text{ord}_{H}^{(X_{0}\sqcup{X_{1}},L_{0}\sqcup{L_{1}})}(\iota([S_{0}]\otimes{[S_{1}]})).
    \]
\end{proposition}

\begin{proof}
    We consider the composition of maps 
    \[
    \phi_{\#}\circ \iota:\skein_{0}^{BN}(X_{0};L_{0})\otimes{\skein_{0}^{BN}(X_{1};L_{1})}\rightarrow{\skein_{0}^{BN}(X_{0}\#X_{1};L_{0}\sqcup L_{1})}.
    \]

    We are done if we are able to define an inverse map $(\phi_{\#}\circ{\iota})^{-1}$ with the desired image when restricted to $\mathcal{R}(X_{0}\# X_{1};L_{0}\sqcup L_{1})$. Let $[(S_{0}\sqcup S_{1},v)]$ be a generator of $\mathcal{R}(X_{0}\# X_{1};L_{0}\sqcup L_{1})$. Then, letting $K_{0}$ and $K_{1}$ denote the input links of $S_{0}$ and $S_{1}$ respectively, we have that $v\in{BN(K_{0}\sqcup K_{1})}$. Since $BN(K_{0})$ and $BN(K_{1})$ satisfy the pairwise vanishing torsion condition, we have the natural isomorphism $\mu^{-1}:BN(K_{0}\sqcup K_{1})\xrightarrow{\cong}{BN(K_{0})\otimes BN(K_{1})}$. Note that the preimage of a generator $[(S_{0}\sqcup S_{1},v)]$ under the map $\phi_{\#}$ is the element in $\skein_{0}^{BN}(X_{0}\sqcup X_{1};L_{0}\sqcup L_{1})$ represented by the filling $(S_{0}\sqcup S_{1},v)$. Letting $\mu^{-1}(v)=y\otimes{z}$, we may then apply the map $\mu^{-1}$ to the input link labels to obtain the element $[(S_{0},y)]\otimes{[(S_{1},z)]}\in{\skein_{0}^{BN}(X_{0};L_{0})\otimes{\skein_{0}^{BN}(X_{1};L_{1})}}$. We note that $\phi_{\#}^{-1}([(S_{0}\sqcup S_{1},v)])$ is $\iota([(S_{0},y)]\otimes[(S_{1},z)])$, hence, the the lasagna $H$ torsion orders of $[(S_{0},y)]\otimes{[(S_{1},z)]}$ and $[(S_{0}\sqcup S_{1},v)]\in{\skein_{0}^{BN}(X_{0}\# X_{1};L_{0}\sqcup L_{1})}$ coincide.
\end{proof}

\begin{remark}
    The submodule $\mathcal{R}(X_{0}\# X_{1};L_{0}\sqcup L_{1})$ contains elements represented by surfaces connected through the 4-manifold connect sum region by cobordisms $L\times{I}$ for some link $L$ provided that $\text{Tor}_{1}^{\mathbb{F}_{2}[H]}(BN(K_{0}\sqcup L),BN(K_{1}\sqcup \overline{L}))$ vanishes. Letting $[S]$ denote the element corresponding to the connected surface and letting $\sum_{i}([S_{0,i}]\otimes{[S_{1,i}]})$ denote the result of neck-cutting along $L\times{I}$, we have that
    \[
    \text{ord}_{H}^{(X_{0}\# X_{1},L_{0}\sqcup L_{1})}([S])=\text{ord}_{H}^{(X_{0}\sqcup X_{1},L_{0}\sqcup L_{1})}(\sum_{i}\iota([S_{0,i}]\otimes{[S_{1,i}]})).
    \]
\end{remark}

We return to the ideas presented in this section in Section \ref{sec:4}.

\subsection{Deformations and decompositions into relative second homology}\label{subsubsec:rank}

In this section, we relate results of Ren--Willis \cite{RenWillis} and Morrison--Walker--Wedrich \cite{mwwsurfaces}. We then describe the structure of the free part of $\skein_{0}^{BN}(X;L)$ using a deformation of the invariant. Throughout the remainder of this work, we let $H_{2}(X)^{L}$ denote the $H_{2}(X,L)$-torsor given by the preimage $\partial^{-1}([L])$, where $\partial$ is the connecting homomorphism in the homology long exact sequence of the pair $(X,L)$.

\begin{definition}
    Let $L$ be a link in $S^{3}$ and let $H(L)$ denote the bigraded rank one free abelian group $\mathbb{F}_{2}$ concentrated in bidegree $(-\text{lk}(L),\text{lk}(L))$, generated by preimage $\partial^{-1}([L])$ in $H_{2}(B^{4},L)$.
\end{definition}

\begin{remark}
The invariant $H(L)$ is functorial for links in $S^{3}$, and is naturally isomorphic to the $\mathfrak{gl}_{1}$ link homology theory $\KhR_{1}$ as bigraded link homology theories. Thus, we let $\skein_{0}^{1}(X;L)$ denote the homological skein module constructed from the $H(\underline{\;\;})$ link homology theory. We note that any surface $S$ representing an element $\alpha\in{H_{2}(X)^{L}}$ yields an element $[S]\in{\skein_{0}^{1}(X;L,\alpha)}$.
\end{remark}

\begin{definition}
    Let $\Sigma=\{\lambda_{1},\lambda_{2}\}$ be a set of deformation parameters in $\mathbb{F}_{2}$. We let $\skein_{0}^{\Sigma}(X;L)$ denote the skein lasagna module built from the link homology theory obtained by replacing $E_{1}$ and $E_{2}$ in Definition \ref{def:frobpair} with $E_{1}(\lambda_{1},\lambda_{2})$ and $E_{2}(\lambda_{1},\lambda_{2})$.
\end{definition}

It is possible to characterize $\skein_{0}^{\Sigma}(X;L)$ in terms of $\skein_{0}^{1}(X;L)$ modules.

\begin{corollary}[Corollary 3.14 -- \cite{mwwsurfaces}]
    Let $\Sigma=\{1,0\}$, and let $C(L,\Sigma)$ denote the set of \emph{colorings} of the components of $L$ by elements of $\Sigma$. Let $L_{c^{-1}(i)}$ denote the sublink of $L$ with color $i$ colored by $c$. There is an isomorphism of $H_{2}(X)^{L}\times{\mathbb{Z}_{q}}\times{\mathbb{Z}_{h}}$-graded $\mathbb{F}_{2}$-vector spaces
    \begin{equation}\label{eq:gl1decomp}
        \skein^{\Sigma}_{0}(X;L)\cong{\bigoplus_{c\in{C(L,\Sigma)}}\skein_{0}^{1}(X;L_{c^{-1}(1)},\mathbb{F}_{2})\otimes{\skein_{0}^{1}(X;L_{c^{-1}(0)},\mathbb{F}_{2}})}
    \end{equation}
\end{corollary}

\begin{remark}
    Let $S$ be an oriented surface in $(X,L)$ with no closed components. For a coloring $c\in{C(L,\Sigma)}$, the element corresponding to $[S]\in{\skein_{0}^{\Sigma}(X;L)}$ under the isomorphism in \eqref{eq:gl1decomp} is the tensor product $[S_{c^{-1}(1)}]\otimes{[S_{c^{-1}(0)}]}$, where $[S_{c^{-1}(\lambda_{i})}]$ denotes the relative second homology class given by the subsurface $S_{c^{-1}(\lambda_{i})}$ whose boundary is the sublink colored $\lambda_{i}$. We can relax the conditions on $S$ to include closed components as long as $S$ remains homologically diverse.
\end{remark}

We then adopt the notation $H_{2}^{L,\times{2}}(X)$ from \cite{RenWillis} to denote the set of relative \emph{double classes} of the pair $(X,L)$. Defined as

\[
    H_{2}^{L,\times{2}}(X):=\{(\alpha_{-},\alpha_{+})\in{H_{2}(X;L)^{2}\;|\;\partial{\alpha_{\pm}}=\sum_{i}\epsilon_{i,\pm}[L_{i}],\;\epsilon_{i,\pm}\in{\{0,1\}},\epsilon_{i,+}+\epsilon_{i,-}=1\}}.
\]

We also require the notion of a \emph{skein category}, defined as follows.

\begin{definition}
    Let $(X,L)$ be a 4-manifold and boundary link pair. The corresponding \emph{skein category} $\mathcal{C}(X;L)$ has skein surfaces of $(X,L)$ as its objects, and, for any $\Sigma_{0}$, $\Sigma_{1}\in{\text{Ob}(\mathcal{C}(X;L))}$ with input data $\{B_{0,i},L_{0,i}\}_{i\in{I}}$ and $\{B_{1,j},L_{1,j}\}_{j\in{J}}$ respectively, the morphisms are 
    \[
        \text{Hom}(\Sigma_{0},\Sigma_{1})=\{([S]:\Sigma_{0}\rightarrow{\Sigma_{1}},f)\}
    \]
    where $[S]$ is the isotopy class of a surface $S$ such that $S|_{\partial{B_{0,i}}}=\Sigma_{0}|_{\partial{B_{0,i}}}$ and $S|_{\partial{B_{1,j}}}=\Sigma_{1}|_{\partial{B_{1,j}}}$, and $f$ is an isotopy class of isotopies $S\cup{\Sigma_{1}}$ to $\Sigma_{0}$.

\end{definition}

For more details, see \cite{RenWillis} Section 2. 
We now study a \emph{localized} version of the invariant $\skein_{0}^{BN}(X;L)$ with $H=1$. Let $BN_{H=1}$ denote the link homology theory obtained by setting $H=1$ in the Bar-Natan Frobenius algebra $A^{BN}$ as in Definition \ref{def:localizedBN}. 
Equivalently, $BN_{H=1}$ is the link homology theory obtained by deforming the $U(2)$-equivariant Frobenius pair by $\Sigma=\{1,0\}$. Note that the skein lasagna module constructed from $BN_{H=1}$ is exactly the deformation $\skein_{0}^{\Sigma}(X;L)$ for $\Sigma=\{1,0\}$, and therefore satisfies analogous properties. Our goal is to prove the following.

\begin{proposition}\label{prop:forfreepart}
    The rank of the free part of $\skein_{0}^{BN}(X;L)$ is bounded below by the dimension of $\skein_{0}^{BN_{H=1}}(X;L)$. Furthermore, there is an isomorphism
    \[
    \Phi:\skein_{0}^{BN_{H=1}}(X;L)\xrightarrow{\cong}{\mathbb{F}_{2}^{H_{2}^{L,\times{2}}(X)}}
    \]
    and there exists a basis for $\skein_{0}^{BN_{H=1}}(X;L)$ of canonical generators that correspond to elements of $H_{2}^{L,\times{2}}(X)$.
\end{proposition}

This proposition follows from the isomorphism in \eqref{eq:gl1decomp}, but for completeness, we reprove this result using the terminology in \cite{RenWillis}. Recall that a \emph{double skein} is an oriented surface $\Sigma\subset{(X,L)}$ and a partition $\Sigma=\Sigma_{+}\cup{\Sigma_{-}}$ of the surface given by a second orientation $\mathfrak{D}$. If the original orientation agrees (resp. disagrees) with $\mathfrak{D}$ on certain components, we denote this subsurface by $\Sigma_{+}$ (resp. $\Sigma_{-}$).

\begin{definition}
    A \emph{canonical Bar-Natan lasagna filling}, denoted $x(\Sigma_{+},\Sigma_{-})$, is a filling defined by labeling each boundary link $L_{i}$ with the canonical generator $x_{\mathfrak{D}}|_{L_{i}}$, and we label the components of $\Sigma_{+}$ with \textbf{a} and $\Sigma_{-}$ with \textbf{b}, where $\textbf{a}=x-1$ and $\textbf{b}=x$. We note that, unlike the skein lasagna module construction for Lee homology, we do not require any re-scaling with this theory, as the discriminant for this localized theory is $1$.
\end{definition}

\begin{lemma}\label{lem:RW-lemmasforH=1} The following items are analogues of results from \cite{RenWillis}. 
    \begin{enumerate}
        \item Let $\Sigma\subset{(X,L)}$ be a skein surface, then every lasagna filling with surface $\Sigma$ in the free part of $\skein_{0}^{BN_{H=1}}(X;L)$ is skein equivalent to a linear combination of canonical Bar-Natan lasagna fillings.
        \item Let $S:\Sigma\rightarrow{\Sigma^{\prime}}$ be a morphism in the category of skeins $\mathcal{C}(X;L)$ that respects double skein structures, then the map induced by $S$ in $\skein_{0}^{BN_{H=1}}(X;L)$ maps $x(\Sigma_{+},\Sigma_{-})$ to $x(\Sigma^{\prime}_{+},\Sigma^{\prime}_{-})$.
    \end{enumerate}
\end{lemma}

\begin{proof}
    The argument in \cite{RenWillis} for the Lee lasagna module adapts immediately to this setting for part (1). For part (2), we make the following observation about the localized theory. Let $S:L\rightarrow{L^{\prime}}$ be a framed cobordism and let $\mathfrak{s}_{o}$ be a canonical generator corresponding to the orientation $o$ of $L$. Then
    \[
    BN_{H=1}(S)(\mathfrak{s}_{o})=\sum_{\mathfrak{D}|_{L}=o}\mathfrak{s}_{\mathfrak{D}|_{L^{\prime}}}
    \]
    where the sum runs over all orientations of $S$ that are compatible with the orientation $o$ on $L$. This follows from \cite{rasinv-genus} using a modification of the Frobenius algebra and $\mathbb{F}_{2}$ coefficients. As a consequence of this, on the tensor product of generators of the boundary links of a skein surface $\Sigma$, we have
    \[
    BN_{H=1}(S)(\otimes_{i}{\mathfrak{s}_{\mathfrak{D}|_{L_{i}}}})=\otimes_{j}\mathfrak{s}_{\mathfrak{D}|_{L^{\prime}_{j}}}+Y
    \]
    where $Y$ consists of terms with incompatible orientations. Thus, the map induced by $S$ on $x(\Sigma_{+},\Sigma_{-})$ as a morphism in $\mathcal{C}(X;L)$ maps $x(\Sigma_{+},\Sigma_{-})\mapsto{x(\Sigma^{\prime}_{+},\Sigma^{\prime}_{-})}$.
\end{proof}

\begin{proof}{(\textit{of Proposition \ref{prop:forfreepart}})} As in \cite{RenWillis}, define a map $\epsilon:\skein_{0}^{BN_{H=1}}(X;L)\rightarrow{\mathbb{F}_{2}^{H_{2}^{L,\times{2}}(X)}}$ by $x(\Sigma_{+},\Sigma_{-})\mapsto{e_{([\Sigma_{+}],[\Sigma_{-}])}}$, the generator of the $([\Sigma_{+}],[\Sigma_{-}])$-th coordinate of $\mathbb{F}_{2}^{H_{2}^{L,\times{2}}(X)}$. Lemma \ref{lem:RW-lemmasforH=1} implies that $\epsilon$ is well-defined, and the remainder of the argument presented in \cite{RenWillis} occurs entirely in the skein category before the application of $BN_{H=1}$, so the result follows. \end{proof}

\begin{remark}
    One may expect these similarities between $\skein_{0}^{Lee}$ and $\skein_{0}^{BN_{H=1}}$, as the localized theory is equivalent to Lee homology \cite{khovanov2020linkhomologyfrobeniusextensions}.
\end{remark}

We now wish to verify that homologically diverse surfaces are non-trivial modulo $H$-torsion in $\skein_{0}^{BN}(X;L)$ with $\mathbb{F}_{2}$ coefficients.

\begin{corollary}\label{cor:divsurfacesforF2}
    Let $S\subset{(X,L)}$ be a homologically diverse surface, then the element $[S]\in{\skein_{0}^{BN}(X;L)}$ is non-trivial modulo $H$-torsion.
\end{corollary}

\begin{proof}
    As in the proof of the analogous result in \cite{mwwsurfaces}, homologically diverse surfaces represent non-trivial generators of $\mathbb{F}_{2}^{H_{2}^{L,\times{2}}(X)}$ which is isomorphic to $\skein_{0}^{BN_{H=1}}(X;L)$ by Proposition \ref{prop:forfreepart}. The result then follows.
\end{proof}

\section{Extending the exotically knotted pairs of surfaces in the 4-ball to other 4-manifolds}\label{sec:4}

We now explain the skein lasagna module interpretation of the results of Hayden in \cite{hayden2023atomic}. For each disk $D$ and $D^{\prime}$ with common boundary $K_{H}$, we define lasagna fillings represented by punctured, time-reversed disks $[-\overset{\circ}{D}]$ and $[-\overset{\circ}{D^{\prime}}]$, with a single input ball with an unknot on the boundary equipped with label $\Phi_{U}^{-1}((1,0))$. The propositions in Remark \ref{prop:hay4.1,4.2}, Lemma \ref{lem:4balllasagnaisZ}, and Lemma \ref{lem:unreducedworks} then imply that these fillings are not equal in $\skein_{0}^{BN}(B^{4};K_{H})$.

\begin{remark}
    Occasionally, we start with a surface $S\subset{(X,L)}$, then discuss the filling $[S]$ \emph{corresponding} to $S$. By this, we mean to take the surface $S$, with the possible addition of some framing-changing input balls, as the lasagna filling representing $[S]$.
\end{remark}

\begin{definition} 
Letting these fillings be denoted $\mathcal{F}_{0}$ and $\mathcal{F}^{\prime}_{0}$ respectively, we define the Bar-Natan lasagna \emph{distinguishing element} $\delta^{L}_{0}=\mathcal{F}_{0}-\mathcal{F}^{\prime}_{0}$. This is precisely the element $\delta_{0}\in{BN(K_{H})}$ defined in Lemma \ref{lem:unreducedworks}, reinterpreted as an element in the Bar-Natan skein lasagna module. Similarly, for the exotic genus $g$ surfaces $F_{g}$ and $F_{g}^{\prime}$ in \cite{hayden2023atomic}, we let $[F_{g}]$ and $[F_{g}^{\prime}]$ denote the corresponding Bar-Natan fillings in $(B^{4},K_{H})$, and define the distinguishing element $\delta^{L}_{g}:=[F_{g}]-[F_{g}^{\prime}]$.
\end{definition}

\begin{remark} As stated in Lemma \ref{lem:4balllasagnaisZ}, when the 4-manifold and boundary link pair is $(B^{4},L)$, skein lasagna modules recover the TQFT from which they were constructed, allowing for the following restatement in terms of internal stabilization distance.
\end{remark}

\begin{corollary}\label{cor:Haydeninterpretation}
    Let $\delta^{L}_{0}$ and $\delta^{L}_{g}$ be defined as above, then for all $g\geq{0}$, we have that $\text{ord}^{(B^{4},K_{g})}_{H}(\delta^{L}_{g})>1$.
\end{corollary}

\begin{remark} Unfortunately, the Bar-Natan homology package, unlike link Floer homology (in reference to the results of Guth in \cite{guth2022exotic}), does not immediately produce exotically knotted pairs of surfaces with arbitrarily large internal stabilization distance in $B^{4}$. This leads to some difficulty in producing ``one is not enough'' results in other 4-manifolds through a direct application of the gluing map in \ref{prop:BNneckcut}. However, this difficulty can be overcome by restricting to certain types of homologically diverse surfaces. We say that a surface $S\subset{(X,L)}$ has \emph{local genus} $g$ if there exists a 4-ball $B$ embedded in $\text{int}(X)$ such that $S\cap{B^{4}}$ is a punctured oriented surface of genus $g$.
\end{remark}

\begin{definition}\label{def:primitive}
    For a homologically diverse surface $S\subset{(X,L)}$, a corresponding Bar-Natan lasagna filling is \emph{primitive} if represents a a non-trivial free generator $[S]\in\skein_{0}^{BN}(X;L,\alpha)$ and $H$ does not divide $[S]$.
\end{definition}

\begin{lemma}
    For homologically diverse surface $S\subset{(X,L)}$, a filling $[S]$ is primitive if and only if $S$ has no local genera and $[S]$ is not equivalent in $\skein_{0}^{BN}(X;L)$ to a sum of elements given by fillings each with surfaces having local genera.
\end{lemma}

\begin{proof}
    Suppose first that $S\subset{(X,L)}$ has local $g$-genera and let $B$ denote the corresponding $4$-ball. We can isotope the subsurface $S\cap{B}$ into a collar neighborhood of $\partial{X}$. This isotopy realizes $S$ as the image of the surface $S^{\prime}=S\setminus{(S\cap{B})}\cup_{U}D^{2}$ (the surface obtained by removing the genus $g$ subsurface and replacing it with a disk) under $g$ many $H$-maps, implying that $[S]$ is divisible by $H$. Suppose, alternatively, that the lasagna filling is not primitive, then $[S]=H^{k}\cdot{v}$ for some $v\in{\skein_{0}^{BN}(X;L)}$, $k\in{\mathbb{N}}$. It is sufficient to let $k=1$. Recall that, over $\mathbb{F}_{2}$ coefficients, an $H$ factor is interchangeable with a $T^{2}$ connect summand. The presence of this $H$ factor implies that, for some input ball and link pair $(B_{i},L_{i})$ in $[S]$, the lasagna filling $[S]$ is equivalent to a filling $[S^{\prime}]$ obtained by replacing $(B_{i},L_{i})$ with a slightly smaller $B_{i}^{\prime}\subset{B_{i}}$ such that $S^{\prime}\cap{(B_{i}\setminus{B_{i}^{\prime}})}=(L_{i}\times{I})\# T^{2}$. Thus, we can choose yet another small input ball $B_{i}^{\prime \prime}$ in the $S^{3}\times{I}$ region between $B_{i}$ and $B_{i}^{\prime}$ such that $B_{i}^{\prime \prime}\cap{S^{\prime}}=T^{2}\setminus{pt}$. Therefore, $S$ is equivalent to a filling $S^{\prime}$ with local genera.

    The case where $L=\emptyset$ is similar. If $S$ has local genera, then we may neck-cut (see Figure \ref{fig:2}) to produce a dotted torus embedded in $\text{int}(X)$. After removing the 4-ball, we may isotope the dotted torus to the newly created $S^{3}$ boundary.

    Alternatively, if $S\subset{(X,L)}$ has no local genera but is skein equivalent to a sum of fillings with surfaces with local genera, then we may apply the argument above on each summand.
\end{proof}

We now state the main result allowing for the extension of non-stabilizing exotic knotted pairs of surfaces to other 4-manifolds.

\begin{proposition}\label{prop:main}
    Let $S\subset{(X,L)}$ be a homologically diverse surface with a primitive filling $[S]\in{\skein_{0}^{BN}(X;L)}$, then for any $g\geq{1}$, the surfaces $F_{g}\sqcup{S}$ and $F^{\prime}_{g}\sqcup{S}$ form an exotically knotted pair in $(X\setminus{B^{4}},K_{g}\sqcup L)$, and remain exotic after a single internal stabilization.
\end{proposition}

\begin{proof}
    This follows directly from the definition of primitive fillings and Proposition \ref{prop:BNneckcut}. For simplicity, we present the proof in the case where the boundary link of $X$ is the empty link; the proof is identical in the other case.
    
    Let $S$ be a homologically diverse surface in $(X,\emptyset)$, and let $[S]$ be a corresponding primitive filling. Then the element $[S]$ is a free generator with no $H$ factors in $\skein_{0}^{BN}(X)$. By Corollary \ref{cor:Haydeninterpretation}, we have that $\delta^{L}_{g}$ is non-trivial and $H$-torsion in $\skein_{0}^{BN}(B^{4};K_{g})$. Then, by Proposition \ref{prop:BNneckcut}, we have that the lasagna $H$-torsion order of the element given by $[F_{g}\sqcup S]-[F^{\prime}_{g}\sqcup S]$ is equal to $\text{ord}_{H}^{(X\setminus{B^{4}},K_{g}\sqcup L)}(\iota(\delta_{g}^{L}\otimes{[S]}))$. The element $\iota(\delta_{g}^{L}\otimes{[S]})$ is equal to $\delta_{g}^{L}\otimes{[S]}$ in the isomorphic copy of $\skein_{0}^{BN}(B^{4};K_{g})\otimes{\skein_{0}^{BN}(X;L)}$ in $\skein_{0}^{BN}(B^{4}\sqcup{X};K_{g}\sqcup{L})$, implying $[F_{g}\sqcup S]-[F^{\prime}_{g}\sqcup S]$ is non-trivial with the same lasagna $H$-torsion order as $\delta^{L}_{g}$.
\end{proof}

Before we provide any explicit examples, we need a method to determine if the filling represented by a homologically diverse surface is indeed primitive.

\begin{lemma}\label{lem:primitiveskein}
    Recall the linear map $F:\frac{\mathbb{F}_{2}[H]}{H}\otimes\skein_{0}^{BN}(X;L)\rightarrow{\skein_{0}^{2}(X;L)}$ from Lemma \ref{lem:linmapBNtoKhR}. For a homologically diverse surface $S\subset{(X,L)}$, if the image of it's corresponding filling $F([S])\in{\skein_{0}^{BN}(X;L)}$ is non-zero, then $[S]\in{\skein_{0}^{BN}(X;L)}$ is primitive.
\end{lemma}

\begin{proof}
    Given a filling $\mathcal{F}\in{\skein_{0}^{BN}(X;L)}$, the image $F(\mathcal{F})\in{\skein_{0}^{2}(X;L)}$ is a filling consisting of the same topological data as $\mathcal{F}$, with input link labels obtained by setting $H=0$ in the labels of $\mathcal{F}$. If a surface $S\subset{(X,L)}$ has local genus (or is equivalent in $\skein_{0}^{BN}(X;L)$ to a sum of fillings each with local genus), then the corresponding filling is of the form $H\cdot{v}$ for some $v\in{\skein_{0}^{BN}(X;L)}$ and therefore $F([S])=0\in{\skein_{0}^{2}(X;L)}$.
\end{proof}

\begin{remark}
    We remark now on the steps necessary to produce examples of exotically knotted pairs of surfaces as in Proposition \ref{prop:main}. Suppose that $S\subset{X}$ is an embedded surface in 4-manifold $X$, and suppose that $S$ is homologically diverse. To verify that $S$ admits a primitive filling in $\skein_{0}^{BN}(X;L)$, we must begin with a corresponding filling $[S]_{\KhR_{2}}$ in $\skein_{0}^{2}(X;L;\mathbb{F}_{2})$. We then must check that $[S]_{\KhR_{2}}\neq{0}$ through some pre-existing non-vanishing criterion and that $F^{-1}([S]_{\KhR_{2}})$ is non-empty. If these criteria are met, we may obtain a non-zero element $1\otimes{[S]}$ in $F^{-1}([S]_{\KhR_{2}})$, implying the filling $[S]\in{\skein_{0}^{BN}(X;L)}$ is primitive. We may then consider the $BN$ lasagna filling represented by the surfaces $F_{g}\sqcup S$ and $F_{g}^{\prime}\sqcup S$ in $X\setminus{B^{4}}$, where the common boundary $K_{g}$ is a boundary link on $\partial{B^{4}}$. Proposition \ref{prop:BNneckcut} allows us to identify the fillings $[F_{g}\sqcup S]$ and $[F_{g}^{\prime}\sqcup S]$ with tensor products $[F_{g}]\otimes{[S]}$ and $[F_{g}^{\prime}]\otimes[S]$ respectively, implying that $\delta_{g}^{L}\otimes{[S]}=[F_{g}]\otimes{[S]}-[F_{g}^{\prime}]\otimes{[S]}\neq{0}$ and has lasagna $H$-torsion order equal to that of $\delta_{g}^{L}$.
\end{remark}

\begin{remark}\label{rmk:f+}
    An explicit non-vanishing criterion for $\skein_{0}^{2}(X;\emptyset;\mathbb{F}_{2})$ is not known, only interpretations of other non-vanishing results using $\mathbb{F}_{2}$ coefficients. If such a criterion existed, it would still be necessary to describe the skein surfaces that represent non-zero Bar-Natan lasagna fillings in order to produce explicit examples. We expect that this is possible for $D^{2}$-bundles over $S^{2}$ with nonpositive Euler number. Fixing some $n\leq{0}$, let $D(n)$ denote the corresponding disk bundle and let $S\subset{D(n)}$ denote the sphere with self-intersection $[S]\cdot{[S]}=-n$. Next, let $f_{-}\in{BN(U)}$ denote the label of the -1-framed input unknot on an input ball in $D(n)$. In other words, $f_{-}$ is the image of the birth cobordism composed with the corresponding Reidemeister 1 map. We then construct the skein surface $S^{\prime}$ and filling $[S^{\prime}]$ as 
    \[
        [S^{\prime}]:=(S\setminus{\bigsqcup_{i=1}^{n}D_{i}^{2}},f_{-}\otimes{...}\otimes{f_{-}}).
    \]
    This new skein surface $S^{\prime}$ is given by removing $n$ disks from $S$, and adding input balls, with -1-framed boundary unknots labeled by $f_{-}$.
\end{remark}

\begin{remark}\label{rmk:RWcp2bar} 
    As a consequence of Proposition 1.15 in \cite{RenWillis} and Theorem 2.1(1) \cite{RenLee}, we have the following results for the pair $(\overline{\mathbb{C}P^{2}},\emptyset)$ (the $D^{2}$-bundle over $S^{2}$ with Euler number -1) over field coefficients:
    \[
        \skein_{0,-2p^{2},2p^{2}+2p}(\overline{\mathbb{C}P^{2}};\emptyset;0)\cong{\mathbb{F}},\; (p\geq{0})
    \]
    \[
        \skein_{0,-2p^{2}+2p,2p^{2}-1}(\overline{\mathbb{C}P^{2}};\emptyset;1)\cong{\mathbb{F}},\;(p\geq{1})
    \]

In particular, the empty skein and the sphere of self intersection -1 (with a single input ball with a -1-framed input unknot and label $f_{-}$) represent non-trivial elements in $\skein_{0}^{2}(\overline{\mathbb{C}P^{2}};\emptyset)$ (see \cite{RenWillis} Section 6). We now proceed with the construction of exotic surfaces in $\overline{\mathbb{C}P^{2}}$ using the exotic knotted pairs of Hayden.
\end{remark}

\begin{example}\label{ex:cp2bar}
    Let $X=\overline{\mathbb{C}P^{2}}$ and let $\overline{\mathbb{C}P^{1}}$ denote the sphere of -1 self-intersection. Note that $\overline{\mathbb{C}P^{1}}$ is homologically diverse in $(\overline{\mathbb{C}P^{2}},\emptyset)$, as it represents a generator of $H_{2}(\overline{\mathbb{C}P^{2}})\cong{\mathbb{Z}}$. Let $[\overline{\mathbb{C}P^{1}}]\in{\skein_{0}^{BN}(\overline{\mathbb{C}P^{2}};\emptyset)}$ and $[\overline{\mathbb{C}P^{1}}]_{\KhR_{2}}\in{\skein_{0}^{2}(\overline{\mathbb{C}P^{2}};\emptyset;\mathbb{F}_{2})}$ denote the fillings obtained from $\overline{\mathbb{C}P^{1}}$ as in Remark \ref{rmk:f+} in their respective lasagna modules. Let $f^{BN}_{-}$ and $f^{\KhR_{2}}_{-}$ denote the unknot labels for the skein surfaces of $[\overline{\mathbb{C}P^{1}}]_{BN}$ and $[\overline{\mathbb{C}P^{1}}]_{\KhR_{2}}$ respectively. The natural map $\mu:\frac{\mathbb{F}_{2}[H]}{H}\otimes{BN(\underline{\;\;})}\rightarrow{\KhR_{2}(\underline{\;\;})}$ is given by realizing $\KhR_{2}$ as the homology of the Bar-Natan complex tensored with $\frac{\mathbb{F}_{2}[H]}{H}$ then applying the universal coefficient theorem. Note that, for the homologies of the -1-framed unknot, $\mu(f_{-}^{BN})=f_{-}^{\KhR_{2}}$. Hence, the filling $[\overline{\mathbb{C}P^{1}}]_{BN}$ is mapped to $[\overline{\mathbb{C}P^{1}}]_{\KhR_{2}}$ by the composition of maps

    \begin{equation}\label{eq:composition}
        \skein_{0}^{BN}(\overline{\mathbb{C}P^{2}};\emptyset,1)\xrightarrow{1\otimes{\underline{\;\;}}}{\frac{\mathbb{F}_{2}[H]}{H}}\otimes{\skein_{0}^{BN}(\overline{\mathbb{C}P^{2}};\emptyset,1)}\xrightarrow{F}\skein_{0}^{2}(\overline{\mathbb{C}P^{2}};\emptyset,1)
    \end{equation}

    where $F$ is the linear map from Lemma \ref{lem:linmapBNtoKhR}. Since the filling $[\overline{\mathbb{C}P^{1}}]_{\KhR_{2}}$ is a generator of $\skein_{0}^{2}(\overline{\mathbb{C}P^{2}};\emptyset)$, we have that, $1\otimes{[\overline{\mathbb{C}P^{1}}]_{BN}}\in{\frac{\mathbb{F}_{2}[H]}{H}\otimes{\skein_{0}^{BN}(\overline{\mathbb{C}P^{2}}};\emptyset)}$ is non-zero, as it's image under the linear map from Lemma \ref{lem:linmapBNtoKhR} is the $\KhR_{2}$ filling with an identical skein surface and label given by $f_{-}^{\KhR_{2}}$. This implies, by Lemma \ref{lem:primitiveskein}, that $[\overline{\mathbb{C}P^{1}}]$ is a primitive filling in $\skein_{0}^{BN}(\overline{\mathbb{C}P^{2}};\emptyset)$.

    We may now apply Proposition \ref{prop:BNneckcut} and Proposition \ref{prop:main} for the torsion element $\delta_{g}^{L}$, we have that
    \[
    \text{ord}_{H}^{(B^{4}\sqcup \overline{\mathbb{C}P^{2}},K_{g})}(\iota(\delta_{g}^{L}\otimes{[S]}))=\text{ord}_{H}^{(\overline{\mathbb{C}P^{2}}\setminus{B^{4}};K_{g})}([F_{g}\sqcup \overline{\mathbb{C}P^{1}}]-[F_{g}^{\prime}\sqcup \overline{\mathbb{C}P^{1}}])=\text{ord}_{H}(\delta_{g}^{L})>1
    \]

    This directly implies that the surfaces $F_{g}\sqcup \overline{\mathbb{C}P^{1}}$ and $F^{\prime}_{g}\sqcup \overline{\mathbb{C}P^{1}}$ form an exotically knotted pair in $(\overline{\mathbb{C}P^{2}}\setminus{B^{4}};K_{g})$ that do not become smoothly isotopic after a single internal stabilization.

    We can perform an even simpler maneuver for the element represented by the empty skein. Clearly, the element $[\emptyset]_{BN}$ is mapped to $[\emptyset]_{\KhR_{2}}$ by the composition in \eqref{eq:composition}. Since $[\emptyset]_{\KhR_{2}}\neq{0}$ in $\skein_{0}^{2}(\overline{\mathbb{C}P^{2}};\emptyset;0)$, we have that $[\emptyset]_{BN}$ is non-zero and primitive in $\skein_{0}^{BN}(\overline{\mathbb{C}P^{2}};\emptyset;0)$. Hence, we similarly have that
    \[
    \text{ord}_{H}^{(\overline{\mathbb{C}P^{2}}\setminus{B^{4}};K_{g})}([F_{g}]-[F^{\prime}_{g}])=\text{ord}_{H}(\delta^{L}_{g})>1
    \]
    as desired. This implies that the exotically knotted pairs $(F_{g},F^{\prime}_{g})$ remain exotic and do not stabilize after one internal stabilization after a 2-handle is attached away from the boundary knot $K_{g}$.
\end{example}

Similar arguments may hold for $S^{2}\times{D^{2}}$ and $D^{2}$-bundles over $S^{2}$ using the non-vanishing results of Manolescu-Neithalath \cite{MN22} and Ren-Willis \cite{RenWillis}.

\begin{remark}
    Note that for a pair of 4-manifold and boundary link pairs $(X_{0},L_{0})$ and $(X_{1},L_{1})$, and skein surfaces $S_{0}$ and $S_{1}$ respectively with $K_{0}$ and $K_{1}$ as respective input links, that $\text{Tor}_{1}^{\mathbb{F}_{2}[H]}(BN(K_{0}),BN(K_{1}))=0$ implies that a filling $[S_{0}\#S_{1}]$ is an element of ${\mathcal{R}(X_{0}\#X_{1};L_{0}\sqcup L_{1})}$ as neck-cutting the connect-sum region only produces an additional unknot. Letting $S^{\bullet}$ denote the surface $S$ with a dot, we have that Proposition \ref{prop:main} states
    \begin{align*}
        \text{ord}_{H}^{(X_{0}\#X_{1},L_{0}\sqcup{L_{1}})}([S_{0}\#S_{1}])&=\text{ord}_{H}^{(X_{0}\sqcup X_{1},L_{0}\sqcup L_{1})}([S_{0}^{\bullet}\sqcup S_{1}]+[S_{0}\sqcup S_{1}^{\bullet}]+H[S_{0}\sqcup S_{1}])\\
        &=\text{ord}_{H}([S_{0}^{\bullet}]\otimes{[S_{1}]}+[S_{0}]\otimes{[S_{1}^{\bullet}]}+H([S_{0}]\otimes{[S_{1}]}))
    \end{align*}
\end{remark}

\begin{remark}
    We may conclude then that, if $K$ is a knot that bounds exotic surfaces in $B^{4}$ that induce different maps on Bar-Natan homology, then if $(X,L)$ contains a homologically diverse surface that admits a primitive Bar-Natan filling, the knot $K$ as a \emph{local} link in $X\setminus{B^{4}}$ continues to bound exotic surfaces in the new 4-manifold. Furthermore, if the induced maps remain distinct after multiplication by $H$, the new exotic surfaces remain distinct after an internal stabilization. 
\end{remark}

\begin{remark}
    For a 4-manifold $X$ and a homologically diverse surface $S\subset{X}$ with primitive filling, it may be interesting to consider the effect of \emph{external stabilization} on the double branched covers of pairs $(X,F_{g}\sqcup S)$ and $(X,F_{g}^{\prime}\sqcup S)$. Recall that, for an exotic pair of $4$-manifolds $(X_{0},X_{1})$, an external stabilization is the operation corresponding to taking a connect sum with $S^{2}\times{S^{2}}$. The \emph{external stabilization conjecture} claims the following.

    \begin{conjecture}[\cite{CTCWall}]
        Let $(X_{0},X_{1})$ be an exotic pair of closed, simply-connected smooth 4-manifolds, then $X_{0}\#(S^{2}\times{S^{2}})$ is diffeomorphic to $X_{1}\#(S^{2}\times{S^{2}})$.
    \end{conjecture}

    The double branched cover of an internally stabilized surface corresponds to an external stabilization of its branched double cover, so such a question is natural to consider. We remark that the manifolds dealt with throughout this work generally all have non-empty boundary, and there exist many similar results on stabilization conjectures in the closed and relative case (see \cite{Lin_2023, lin2021familybauerfurutainvariantexotic, konno2022exoticcodimension1submanifolds4manifolds, kang2024stabilizationcontractible4manifolds, guth2022exotic, auckly2023smoothly}). One may also consider the effect of external stabilization on exotically knotted pairs of surfaces. Although the results presented in this paper are incapable of producing results for this conjecture with $S^{2}\times{S^{2}}$, one may at least use Proposition \ref{prop:main} to show pairs of surfaces remain exotic after connect summing with other 4-manifolds, given that the empty skein surface represents a non-zero element in the lasagna module.
\end{remark}

\begin{corollary}
    Suppose that a primitive filling $[S]\in{\skein_{0}^{BN}(X;\emptyset)}$ is skein equivalent to the empty filling $[\emptyset]$, then the exotic knotted pair $(F_{g},F_{g}^{\prime})$ remains exotic after a single internal stabilization after connect summing with $X$.
\end{corollary}

\begin{proof}
    This is an application of Proposition \ref{prop:main}.
\end{proof}

\begin{remark}
    Toward a result for a pair of \emph{closed} exotic surfaces, we offer the following perspective. Let $(\Sigma_{1},\Sigma_{2})$ and $(\Sigma^{\prime}_{1},\Sigma^{\prime}_{2})$ be a pair of exotic knotted pairs of surfaces in $(X_{1},L)$ and $(X_{2},L)$ respectively. One may consider a different kind of 
    gluing map for a boundary connect sum
    \[\langle{\;,\;}\rangle:\skein_{0}^{BN}(X_{1}\sqcup{X_{2}};L\sqcup{L})\rightarrow{\skein_{0}^{BN}(X_{1}\natural X_{2};\emptyset)}
    \]
    that maps pairs of disjoint lasagna fillings to the filling obtained by gluing along their common $L$ boundary: $\langle{[F_{1}],[F_{2}]}\rangle=[F_{1}\cup_{L}F_{2}]\in{\skein_{0}^{BN}(X_{1}\natural X_{2};\emptyset)}$. The hope is to produce exotic knotted pairs such that $\langle{[\Sigma_{1}],[\Sigma_{1}^{\prime}]}\rangle\neq{\langle{[\Sigma_{2}],[\Sigma_{2}^{\prime}]\rangle}}$. 
    Stating anything outside of information about the tri-degree of such a map is currently beyond the state of the art.
\end{remark}

%%%% bibliography
%%%%%%%%%%%%%%%%%%%%%%%%%%%%%%%%%%%%%%%%%%%%%%%%%%%
\bibliographystyle{alpha}
\bibliography{main}

\end{document}